\documentclass{amsart}

\usepackage{amssymb}
\usepackage{amsthm}
\usepackage{amsmath} 
\usepackage{amsbsy}
\usepackage[all]{xy}
\usepackage{bm}
\usepackage{hyperref}
\usepackage{tikz}
\newtheorem*{theorem}{Theorem}

\newtheorem*{corollary}{Corollary}

\newcommand{\abs}[1]{\lvert #1 \rvert}

\newcommand{\di}{\operatorname{div}}

\begin{document}

\title[]{A variation on the Donsker-Varadhan inequality for the principal
  eigenvalue} 

\thanks{The research of J.L. was supported in part by the National Science Foundation under award DMS-1454939. S.S. was supported in part by the Institute
of New Economic Thinking under grant \#INO15-00038.}

\author[]{Jianfeng Lu}
\address[Jianfeng Lu]{Department of Mathematics, Department of Physics, and Department of Chemistry,
Duke University, Box 90320, Durham NC 27708, USA}
\email{jianfeng@math.duke.edu}

\author[]{Stefan Steinerberger}
\address[Stefan Steinerberger]{Department of Mathematics, Yale University, New Haven, CT 06510, USA}
\email{stefan.steinerberger@yale.edu}

\begin{abstract} The purpose of this short paper is to give a variation on the classical Donsker-Varadhan inequality, which bounds the first
eigenvalue of a second-order elliptic operator on a bounded domain $\Omega$ by the largest mean first exit time of
the associated drift-diffusion process via
 $$\lambda_1 \geq \frac{1}{\sup_{x \in \Omega} \mathbb{E}_x \tau_{\Omega^c}}.$$
Instead of looking at the mean of the first exit time, we study quantiles: let $d_{p, \partial \Omega}:\Omega \rightarrow \mathbb{R}_{\geq 0}$
be the smallest time $t$ such that the likelihood of exiting within that time is $p$, then
$$\lambda_1 \geq \frac{\log{(1/p)}}{\sup_{x \in \Omega} d_{p,\partial \Omega}(x)}.$$
Moreover, as $p \rightarrow 0$, this lower bound converges to $\lambda_1$.
\end{abstract}
\maketitle

\section{Introduction}
We consider, for open and bounded $\Omega \subset \mathbb{R}^n$,  solutions of the equation
\begin{align}
 -\mbox{div}(a(x) \nabla u(x)) + \nabla V \cdot \nabla u &= \lambda u  &&  \text{in }\Omega \\
u &= 0. && \text{on }\partial \Omega \nonumber
\end{align}
Estimating the smallest possible value of $\lambda_1$ for which the equation has a solution
is a problem of fundamental importance. Finding upper bounds is, in many instances, rather
straightforward by testing with a family of functions -- finding lower bounds is substantially more
difficult. An important conceptual leap is due to Donsker \& Varadhan, who take 
\begin{equation}
 L =  -\mbox{div}(a(x) \nabla u(x)) + \nabla V \cdot \nabla u
\end{equation}
and take $-L$ as the infinitesimal generator of a drift-diffusion process (here and in all
subsequent steps we always assume sufficient
regularity on both the operator and the domain). The maximum mean
exit time then serves as a lower bound of the first eigenvalue $\lambda_1$ (we use the formulation from \cite{BdH}).

\begin{theorem}[Donsker-Varadhan \cite{dv1, dons}, CPAM 1976] 
\begin{equation}
  \lambda_1 \geq \frac{1}{\sup_{x \in \Omega} \mathbb{E}_x \tau_{\Omega^c}}.
\end{equation}
\end{theorem}
\begin{proof}
The proof is simple: note that $w(x) = \mathbb{E}_x \tau_{\Omega^c}$ solves the equation
  \begin{align}
    -\mbox{div}(a(x) \nabla w(x)) + \nabla V \cdot \nabla w &= 1  \qquad \mbox{in}~\Omega \\
    w &= 0 \hspace{21pt} \mbox{on} ~\partial \Omega. \nonumber
  \end{align}
We also observe that, by definition, the first eigenfunction $u(x)$ solves
  \begin{align}
    -\mbox{div}(a(x) \nabla u(x)) + \nabla V \cdot \nabla u &= \lambda_1 u  \qquad \mbox{in}~\Omega \\
    u &= 0 \hspace{31pt} \mbox{on} ~\partial \Omega. \nonumber
  \end{align}
This implies, by linearity,
\begin{equation}
   -\mbox{div}\left(a(x) \nabla   \left[\lambda_1 w(x) \max_{x\in\Omega} \abs{u(x)} - u(x)\right] \right) + \nabla V \cdot \nabla \left(  \lambda_1w(x) \max_{x\in\Omega} \abs{u(x)} - u(x)\right) \geq 0.
\end{equation}
The maximum principle then implies
  \begin{equation}
\lambda_1 w(x) \max_{x\in\Omega} \abs{u(x)} - u(x) \geq 0,
  \end{equation}
which yields
\begin{equation}
 \lambda_1 w(x)  \geq \frac{u(x)}{\max_{x\in\Omega} \abs{u(x)}}
\end{equation}
from which we obtain, by setting $x$ so that $|u|$ assumes its maximum,
  \begin{equation}
    \lambda_1 \max_{x \in \Omega}{ w(x)} \geq 1.
  \end{equation}
\end{proof}

The result can be interpreted in two ways: if, perhaps by symmetry
considerations, it is possible to roughly predict the location that
maximizes the mean first exit time, then the result allows for lower
bounds on the eigenvalue $\lambda_1$ and, conversely, knowledge about
the eigenvalue $\lambda_1$ guarantees the existence of points in the
domain for which the mean first exit time is `large'. Among other
applications, the Donsker-Varadhan estimate is crucially used in the
potential theoretic analysis of metastability in \cite{BdH, BEGK, BGK}
(see Lemma 2.1 in \cite{BGK} where the Lemma is quoted and an
improvement in Lemma 2.2 in the same paper) and in Markov state models
(see e.g., \cite{schutte} and references therein). We shall not focus
too much on the minimal regularity of $L$: the reader may assume that
$a(x)$ is uniformly elliptic and both $a(x)$ and $V(x)$ are smooth; in
practice, the results will hold in much rougher situations and only
relies on the Feynman-Kac formula being applicable (which even allows
moderate singularities in $V$). Moreover, the arguments are versatile
enough to be applicable to Graph Laplacian on Markov chain; the
changes are completely obvious changes of symbols and will not be
detailed in this paper.

\section{The Result}
The Donsker-Varadhan inequality is based on the mean value of the first exit time. We will work
with quantiles of that distribution instead: for fixed $0 < p < 1$, we define the diffusion distance 
to the boundary $d_{p, \partial \Omega}: \Omega \rightarrow \mathbb{R}_{+}$ implicitly as the smallest
number
\begin{equation}
 \mathbb{P}\left(\mbox{first exit time} \geq d_{p, \partial \Omega}(x_0)\right) \leq p,
\end{equation}
where the probability is taken over drift-diffusion processes generated by $-L$ and started in $x_0$.
Our main result is that there is a natural relation between that quantity and the smallest eigenvalue
$\lambda_1$ of the differential operator.
\begin{theorem} We have
\begin{equation}
 d_{p,\partial \Omega}(x) \geq \frac{1}{\lambda} \log{\left(\frac{1}{p} \frac{|u(x)|}{\|u\|_{L^{\infty}(\Omega)} }\right)}.
\end{equation}
\end{theorem}
We are not aware of this result being in the literature. Related
statements seem to have first appeared in \cite{manas, stein1}, a
discrete analogue was given by Cheng, Rachh and the second author in
\cite{xiu}. In most cases, the definition may be simplified as $ \mathbb{P}\left(\mbox{first exit time} \geq d_{p, \partial \Omega}(x_0)\right) = p,$
however, the definition above also covers time-discrete processes on Markov chains with absorbing
states where a similar estimate can be easily obtained (we leave the details to the reader).
\begin{corollary}[Donsker-Varadhan for Quantiles]
  \begin{equation}
    \lambda_1 \geq \frac{\log (1/p)}{\sup_{x \in \Omega} d_{p,\partial \Omega}(x)}.
  \end{equation}
Moreover, the right-hand side converges to $\lambda_1$ as $p \rightarrow 0$.
\end{corollary}

We observe two major differences that become relevant when 
estimating $d_{p,\partial \Omega}(x)$ with a Monte Carlo method:
\begin{enumerate}
\item
instead of having to compute a mean (which, especially for heavy-tail
distributions, can be difficult), it suffices to estimate the
likelihood of exiting within a fixed time $t$. The desired outcome is
a Bernoulli variable with likelihood $p$ -- the problem thus reduces
to estimating the parameter in a $\left\{0,1\right\}$ Bernoulli
distribution and adjusting time $t$, which is more stable.
\item By decreasing the value of $p$, the result can be arbitrarily
refined -- the difficulty being that estimating the parameter becomes
more computationally costly as $p \rightarrow 0$, as one needs more
simulations to ensure that there are enough samples in the $p-$th
quantile to give a stable estimation of the Bernoulli parameter. In practice, the available amount of computation will impose a
restriction on the value of $p$ that can be reasonably estimated with a
certain degree of confidence.
\end{enumerate}

\section{Proofs}
\subsection{Proof of the Theorem.}
\begin{proof}  We assume w.l.o.g. that $u(x) > 0$ and defube the parameter $0 < \delta < 1$ implicitly via  $\delta \|u\|_{L^{\infty}} = u(x)$. We
use $\omega(t)$ to denote drift-diffusion process (associated to the Feynman-Kac formula) started in $x$ and running up to time $t$. Since 
 \begin{equation}
   -\mbox{div}(a(x) \nabla u(x)) + \nabla V \cdot \nabla u = \lambda_1 u,
\end{equation}
we have that
\begin{equation}
 u(x) = e^{\lambda t}      \mathbb{E}_{x}\left(u(\omega(t))  \right)
\end{equation}
with the convention that $u(\omega(t))$ is 0 if the drift-diffusion processes leaves $\Omega$
at some point in the interval $[0,t]$. Let now $t= d_{p,\partial \Omega}(x)$, in which case we see that
 \begin{equation}
 \mathbb{E}_{x}\left(u(\omega(t))  \right) \leq  p \|u\|_{L^{\infty}}  + (1-p)0.
\end{equation}
Altogether, we obtain
\begin{equation}
\delta \|u\|_{L^{\infty}} = u(x) =  e^{\lambda d_{p,\partial \Omega}(x)}  \mathbb{E}_{x}\left(u(\omega(t))  \right) \leq e^{\lambda  d_{\partial \Omega}(x)}  p \|u\|_{L^{\infty}} 
\end{equation}
from which the statement follows.
\end{proof}

\subsection{Proof of the Corollary.}
\begin{proof}  It remains to show that the lower bound is asymptotically sharp as $p \rightarrow 0^+$.
Let $x \in \Omega$ be arbitrary and let $\delta_x$ be the Dirac distribution centered at $x$. We are
interested in the long-time behavior of applying the drift-diffusion process to these initial conditions;
denoting the eigenpairs of the differential operator by $(\lambda_k, \phi_k)$, we can use the spectral
theorem (see e.g. \cite{stroock})  to estimate
\begin{equation}
 \int_{\Omega}{ e^{(\di(a(x)\nabla \cdot) - \nabla V \cdot \nabla \cdot)t} \delta_x dz} = \int_{\Omega}{ \sum_{k=1}^{\infty}{ e^{-\lambda_k t} \left\langle \delta_x, \phi_k\right\rangle \phi_k(z)} dz}.
\end{equation}
The spectral gap implies that, as $t \rightarrow \infty$,
\begin{equation}
 \int_{\Omega}{ \sum_{k=1}^{\infty}{ e^{-\lambda_k t} \left\langle \delta_x, \phi_k\right\rangle \phi_k(z)} dz} = \phi_1(x) e^{-\lambda_1 t}\int_{\Omega}{\phi_1(z) dz} + o\left(e^{-\lambda_1 t}\right).
\end{equation}
This means that, asymptotically as $t \rightarrow \infty$, the survival probability is maximized by starting in the point in which the first eigenfunction assumes
a global maximum. Conversely, the case $p \rightarrow 0$ is equivalent to the case $t \rightarrow \infty$ and by locating $x$ in the point $x_0 \in \Omega$, where the ground state
assumes its maximum, we get that
\begin{equation}
 \sup_{x \in \Omega} d_{p,\partial \Omega}(x) = (1+o(1)) d_{p, \partial \Omega}(x_0).
\end{equation}
The computation above shows, as $p \rightarrow 0$ and $t \rightarrow \infty$
\begin{equation}
 e^{(-\di(a(x)\nabla \cdot) + \nabla V \cdot \nabla \cdot)t} \delta_{x_0} = (1+o(1)) \| \phi_1\|_{L^{\infty}} e^{-\lambda_1 t} \int_{\Omega}{\phi_1(x) dx}.
\end{equation}
This implies nontrivial bounds on the logarithm of the survival probability
\begin{equation}
 \log{\mathbb{P}\left(\mbox{first exit time} \geq t\right)} = - (\lambda_1+o(1)) t.
\end{equation}
Then, by definition,
\begin{equation}
\log{p} = \log{\mathbb{P}\left(\mbox{first exit time} \geq d_{p, \partial(\Omega)}(x_0)\right)} =  - (\lambda_1+o(1)) d_{p, \partial \Omega}(x_0)
\end{equation}
and this then implies
\begin{equation}
(1+o(1)) \lambda_1 \sup_{x \in \Omega}{ d_{p,\partial \Omega}(x)} = \log{(1/p)}.
\end{equation}
\end{proof}
The main idea of the argument is that $p \rightarrow 0$ naturally corresponds to $t \rightarrow \infty$. The 
spectral theorem implies that long-time asymptotics is essentially given by the first eigenvalue 
and the first eigenfunction via
\begin{equation}
 e^{-t L} f \sim   e^{-\lambda_1 t} \left\langle f, \phi_1 \right\rangle \phi_1
\end{equation}
and this is how we indirectly obtain estimates on $\lambda_1$. This also suggests that it might perhaps
be possible to obtain estimates on the convergence speed depending on the spectral gap.

\section{Numerical Examples}

\subsection{Unit interval.} A toy example is given by
\begin{align}
 -\Delta u &= \lambda u \qquad    \text{in }[0,1] \\
u(0) &= 0 = u(1). \nonumber
\end{align}
The ground state is $u(x) = \sin{\pi x}$ and $\lambda_1 = \pi^2 \sim 9.86..$ -- the Donsker-Varadhan estimate
requires us to solve $-\Delta w = 1$, which easily gives $w(x) = x/2 - x^2/2$ and from which we get the
lower bound $ \lambda \geq 8$.
In comparison, our bound for various values of $p$ are
\vspace{5pt}
\begin{center}
\begin{tabular}{ l c c c c c| c  r }
  $p$ & $1/2$ & $1/4$ & $10^{-1}$ & $10^{-2}$  & $10^{-8}$ & Donsker-Varadhan \\
\mbox{lower bound} &$7.28$ & $8.40$ & $8.92$  & $9.39$ & $9.74$ & $8$
\end{tabular}
\end{center}
\vspace{5pt}

\subsection{Unit interval with a quadratic potential.} 
Let us consider a $1D$ example with a quadratic potential
$V = \frac{1}{2} x^2$ on $[-1, 1]$:
\begin{align}
 -\Delta u + x \nabla u &= \lambda u \qquad    \text{in }[-1,1] \\
u(-1) &= 0 = u(1). \nonumber
\end{align}
The ground state is $u(x) = 1 - x^2$ with $\lambda = 2$. The mean first exit time $w$ solves 
\begin{equation}
  -\Delta w + x \nabla w = 1 \qquad \text{in }[-1,1]
\end{equation}
with Dirichlet boundary condition. Solving the equation by central
difference scheme with mesh size $h = 10^{-4}$ yields the
Donsker-Varadhan estimate $\lambda \geq 1.678$. To use our bound for
various values of $p$, we simulate $10^4$ paths using an Euler-Maruyama
scheme with time step size $t = 10^{-4}$ starting at the origin
(thanks to the symmetry), the following lower bounds are obtained.
\vspace{5pt}
\begin{center}
\begin{tabular}{ l  c c c c c  |c  r }
    $p$ & $0.5$ & $0.3$ & $0.2$ & $0.1$  & $0.05$ &  Donsker-Varadhan \\
\mbox{lower bound} & $1.522$ & $1.675$ &  $1.740$  & $1.799$ & $1.834$ & $1.678$
\end{tabular}
\end{center}
\vspace{5pt}

\subsection{Unit disk.} Finally, we estimate the ground state of the Laplacian on the unit disk in $\mathbb{R}^2$, which is given by the first nontrivial zero of the Bessel function $ \lambda_1 \sim 2.40\dots$ while the Donsker-Varadhan estimate gives 
\begin{equation}
w(x) = 1/2-(x^2 + y^2)/2 \quad \mbox{and thus} \quad   \lambda_1 \geq 2.
\end{equation}
Suppose we could not solve
any of these equations in closed form (as is usually the case): using the symmetry of the domain, it suffices to take Brownian
motion started in the origin. Discrete Brownian motion with step size (in time) $t=10^{-4}$ and $10^4$ paths give the
following estimates for a lower bound on $\lambda_1$

\vspace{5pt}
\begin{center}
\begin{tabular}{ l  c c c c c  |c  r }
    $p$ & $0.5$ & $0.4$ & $0.3$ & $0.2$  & $0.1$ &  Donsker-Varadhan \\
\mbox{lower bound} & $1.68$ & $1.85$ & $2.04$  & $2.19$ & $2.37$ & $1.96$
\end{tabular}
\end{center}

\end{document}